\documentclass{amsart}
\usepackage{graphicx} 
\usepackage[mathscr]{eucal}

\theoremstyle{plain}
\newtheorem{prop}{Proposition}[section]

\newtheorem{conj}[prop]{Conjecture}
\newtheorem{lemm}[prop]{Lemma}

\newtheorem{thm}[prop]{Theorem}

\theoremstyle{definition}
\newtheorem{defn}[prop]{Definition}

\newtheorem{rem}[prop]{Remark}

\def\mcg#1;#2{\Gamma_{#1,#2}}
\def\fg#1;#2{\Pi_{#1,#2}}
\def\tb#1;#2{\mathscr{K}_{\frac{#1}{#2}}}

\begin{document}

\title[A new property of quasi-alternating links]
{A new property of quasi-alternating links}


\author{Khaled Qazaqzeh}
\address{Department of Mathematics\\ Yarmouk University\\
Irbid, Jordan, 21163}
\email{qazaqzeh@yu.edu.jo}
\urladdr{http://faculty.yu.edu.jo/qazaqzeh}
\author{Balkees Qublan}
\address{Department of Mathematics\\ Yarmouk University\\
Irbid, Jordan, 21163}
\email{balkesqublan@yahoo.com}
\author{Abeer Jaradat }
\address{Department of Mathematics\\ Yarmouk University\\
Irbid, Jordan, 21163}
\email{abeerjaradat29@yahoo.com}

\thanks{The first two authors are supported by a grant from Yarmouk University.}
\date{16/05/2012}

\begin{abstract}
We show that the crossing number of any link that is known to be quasi-alternating is less than or equal to its determinant. Based on this, we conjecture that the crossing number of any quasi-alternating link is less than or equal to its determinant. Thus, if this conjecture is proved then it would give an easier obstruction for quasi-alternateness than the ones already known.
\end{abstract}

\maketitle

\section{introduction}

Quasi-alternating links were defined in \cite[Definition.\,3.9]{OS} as natural generalization of alternating links. It is well-known already that they share many properties with alternating links. Unfortunately, the recursive definition makes it impossible to show that a given link is not quasi-alternating.

\begin{defn}
The set $\mathcal{Q}$ of quasi-alternating links is the smallest set
satisfying the following properties:
\begin{itemize}
	\item The unknot belongs to $\mathcal{Q}$.
  \item If $L$ is link with a diagram containing a crossing $c$ such that 
\begin{enumerate}
\item both smoothings of the diagram of $L$ at the crossing $c$, $L_{0}$ and $L_{1}$ as in figure \ref{figure} belong to $\mathcal{Q}$, and 
\item $\det(L_{0}), \det(L_{1}) \geq 1$,
\item $\det(L) = \det(L_{0}) + \det(L_{1})$; then $L$ is in $\mathcal{Q}$ and in this case we say $L$ is quasi-alternating at the crossing $c$.
\end{enumerate}
\end{itemize}
\end{defn}

\begin{figure} [h]
 \includegraphics [width = 200pt, height =50pt] {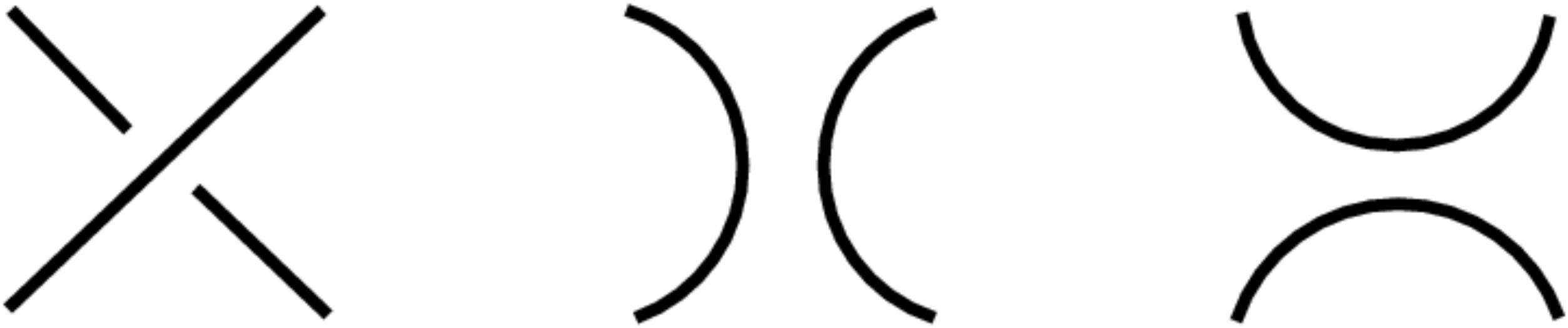}
 \caption{The link $L$ at the crossing $c$ and its smoothings $L_{0}$ and $L_{1}$ respectively.}\label{figure}
\end{figure}

The following is a list of properties hold for quasi-alternating links and hence a list of obstructions for a given link to be quasi-alternating that was first appeared in \cite[Page.\,1]{G}:
\begin{enumerate}
\item the branched double-cover of any quasi-alternating link
is an $L$-space \cite[Proposition.\,3.3]{OS};
\item the space of branched double-cover of any
quasi-alternating link bounds a negative definite $4$-manifold $W$ with $H_{1}(W) = 0$ \cite[Proof of Lemma.\,3.6]{OS};
\item the $\mathbb{Z}/2\mathbb{Z}$ knot Floer homology group of any
quasi-alternating link is thin \cite[Theorem.\,2]{MO};
\item the reduced ordinary Khovanov homology group of any
quasi-alternating link is thin \cite[Theorem.\,1]{MO};
 \item the reduced odd Khovanov homology group of any
quasi-alternating link is thin \cite[Remark after Proposition.\,5.2]{ORS}.

\end{enumerate}

One of the basic well-known properties of alternating links is $c(L) \leq  \det(L)$ as proved in Proposition \ref{basic}. We conjecture that this property holds for
quasi-alternating links. In support of this conjecture, we show this property for all links that are known to be quasi-alternating. In particular, we provide a table of knots up to 11 crossings that don not satisfy the above inequality and known to be not quasi-alternating.

If this conjecture is true, then we obtain a new property and
an easy obstruction for a link to be quasi-alternating. Moreover, this property will solve Conjecture 3.1 in \cite{G} that states that there are finitely many
quasi-alternating links of a given determinant. 

Finally, we provide two tables one of knots of 12 crossings, and the second one of links of 9 crossings or less that we think that they are not quasi-alternating based on our Conjecture \ref{main}.

\section{Main Results}

In this section, we show that the crossing number of any link that is known to be quasi-alternating is less than or equal to its determinant. Alternating links are quasi-alternating links as a result of \cite[Lemma.\,3.2]{OS}. Our result holds for the alternating links by the following proposition, but before that we state a well-known fact in graph theory.
\begin{rem}\label{graph}
The number of edges is less than or equal to the number of spanning trees in a connected graph with no loops and no bridges.
\end{rem}
\begin{prop}\label{basic}
If $L$ is an alternating link then $c(L) \leq \det(L)$.
\end{prop}

\begin{proof}
Take a reduced alternating diagram of $L$, then the Tait's graph is connected and has no loops and no bridges. It is well-known that $c(L)$ represents the number of edges and $\det(L)$ represents the number of spanning trees of such Tait's graph. Finally, the result follows by Remark \ref{graph}.
\end{proof}

\begin{thm}\label{twist}
Let $L$ be a quasi-alternating link at some crossing $c$ with \[
c(D) \leq \det(L),
\] where $ c(D)$ is the crossing number of the quasi-alternating diagram $D$ of the link $L$.
Then the link $L'$ obtained by replacing the crossing $c$ in $L$ with an alternating tangle satisfies the inequality
\[
c(L') \leq \det(L').
\] 
\end{thm}
\begin{proof}
We show the result for integer tangles first. For $n = 1$ the result follows by assumption since $c(L) \leq c(D)$. We let $L'$ be the link obtained by replacing the crossing $c$ in $L$ with an alternating integer tangle of length $k$ and without loss of generality we can assume that is as the crossing in figure \ref{figure}. As a result of \cite[Theorem.\,2.1]{CK}, $L'$ is a quasi-alternating link at any crossing of the integer tangle. Now, we smooth all the crossings of the tangle except one of the quasi-alternating diagram of $L'$ to obtain \[
\det(L') = \det(L) + |k - 1|\det(L_{1}) \geq c(D) + |k - 1| = c(D') \geq c(L'),\]
where $\det(L_{1}) \geq 1$ since it is quasi-alternating. Finally, we repeat this process to obtain the result for rational tangles since after each step we obtain the initial assumption of the theorem. In particular, $c(D') \leq \det(L')$ after each step as above. 
\end{proof}

We state the following well-known results to be used later.
\begin{prop}\label{det}
If $K_{1}, \ldots, K_{n}$ are links, then 
\[
\det(K_{1}\# \ldots \# K_{n}) = \prod_{i = 1}^{n}\det(K_{i}).
\]
\end{prop}

\begin{proof}
The result follows since the Jones polynomial is multiplicative under connected sum of links.
\end{proof}

\begin{prop}\label{crossing}\cite{Ka},\cite{LM},\cite{M},\cite{T}
If $K_{1}, \ldots, K_{n}$ are alternating links, then  
\[
c(K_{1}\# \ldots \# K_{n}) = \sum_{i = 1}^{n}c(K_{i}).
\]
\end{prop}

\begin{lemm}\label{simple}
If $x_{1}, \ldots, x_{n}$ are positive integers, then 
\[
\prod_{i=1}^{n} x_{i} \geq  \sum_{i = 1, x_{i} > 1}^{n} x_{i}.
\]
\end{lemm}
\begin{proof}
The result is clear if $x_{1} = \ldots = x_{n} = 1$. To consider the other cases, we reorder such that $x_{1}\geq x_{2} \geq \ldots \geq x_{m} > 1$ and $x_{m+1} = \ldots = x_{n} = 1$.  We use induction on $m$ and note that the result holds for $m = 1$.
Now we assume that the result holds for $m - 1$ and we want to show that the result holds for $m$.
\[
\prod_{i = 1}^{m} x_{i} = x_{m} \prod_{i = 1}^{m - 1}x_{i} \geq x_{m} \sum_{i = 1}^{m-1} x_{i}\geq x_{m} + \sum_{i = 1}^{m - 1} x_{i} = \sum_{i = 1}^{m} x_{i}.
\]
\end{proof}
\begin{prop}\label{new}
For $p_{1},\ldots, p_{n}, q \geq 2$ and $ q > \min \{ p_{1}, \ldots, p_{n}\}$, we have 
\[
c(L) \leq \det(L),
\]
where $L$ is the pretzel link $P(p_{1},\ldots, p_{n}, -q)$.
\end{prop}

\begin{proof}
For $n = 1$ the statement is true by Proposition \ref{basic} since the pretzel link is the torus alternating link $T(2, p-q)$ if $p \neq q$.
Now we show that the result holds for $n \geq 2$. Let $ L' = P(p_{1},\ldots, p_{n - 1}, p_{n},  -q)$ where $ q > \min \{p_{1}, \ldots, p_{n - 1} \}$, then $L'$ is quasi-alternating as a result of \cite[Theorem.\,3.2]{CK} and later by \cite[Theorem.\,1.4 and Proposition.\,2.2]{G}.

We consider $L = P(p_{1},\ldots, p_{n - 1}, 1 , -q)$. The authors of \cite{CK} show that $L$ is quasi-alternating at the only crossing in the $n$-th tassel in \cite[proof of Theorem.\,3.2]{CK}. Therefore, we obtain 
\begin{align*}
\det(L) = \det(L_{0}) + \det(L_{1}) & =  \det(T(2,p_{1})\# \ldots \# T(2,p_{n - 1}) \# T(2, -q))  +  \det(L_{1})\\ & = q \prod_{i = 1}^{n - 1} p_{i} + \det(L_{1}) 
 \geq \sum_{i=1}^{n-1} p_{i} + q  + 1 = c(L_{0}) + 1 =c(D) \geq c(L),
\end{align*}
where $ L_{1} = P(p_{1},\ldots, p_{n - 1}, -q)$ which is quasi-alternating by the induction hypothesis and $D$ is the pretzel diagram of $L$.
Finally the result follows by applying Theorem \ref{twist} on the link $L$ on the crossing in the $n$-th tassel.

\end{proof}

\begin{rem}
Another proof of above proposition can be given by induction on $n$ and using the famous formulas for the determinant and the crossing number of the given pretzel link. 
\end{rem}
The relevant notation and facts concerning Montesinos and pretzel links in this paper follows \cite[Section 3.2]{OSS}.
\begin{lemm}\label{just}
If $n \geq 0, k \geq 1$ with $ n + k \geq 2$, all $p_{i} \geq 2$ and all $q_{j} \geq 3$, then $\det(M(0;(p_{1},1),\ldots , \\ (p_{n},1), (q_{1}, q_{1}-1) , \ldots, (q_{k}, q_{k}-1)) \geq \sum _{i =1}^{n} p_{i} + \sum_{j =1}^{k} q_{j} + k + 2$.
\end{lemm}
\begin{proof}
We let $L = M(0;(p_{1},1),\ldots, (p_{n},1), (q_{1}, q_{1}-1) , \ldots, (q_{k}, q_{k}-1))$. It is clear that the corresponding Montesinos link is alternating. Therefore, the determinant is equal to the number of spanning trees of the Tait graph. Thus
\begin{align*}
\det(L) = \prod_{i=1}^{n}p_{i}\prod_{j=1}^{k}(q_{j} + 1) & \geq \prod_{i=1}^{n}p_{i}(\prod_{j=1}^{k}q_{j} + k)\\ & = \prod_{i=1}^{n}p_{i}\prod_{j=1}^{k}q_{j} + (\prod_{i=1}^{n}p_{i})k \geq \sum _{i =1}^{n} p_{i} + \sum_{j =1}^{k} q_{j} + k + 2,
\end{align*}
where the last inequality follows by Lemma \ref{simple}.
\end{proof}
\begin{thm}
The crossing number of any pretzel quasi-alternating link is less than or equal to its determinant.
\end{thm}

\begin{proof}

The complete characterization of quasi-alternating pretzel links is given in \cite[Theorem.\,1.4]{G} which states that the pretzel link 
$L  = P(e;p_{1}, \ldots, p_{n}, -q_{1}, \ldots, -q_{m}) = M(e;(p_{1},1), \ldots$ $, \\ (p_{n},1), (q_{1}, -1) , \ldots, (q_{m}, -1))$ 
with $e, n, m \geq 0$, all $p_{i} \geq 2$, and all $q_{j} \geq 3$ is quasi-alternating iff 
\begin{enumerate}
\item $e > m - 1$;
\item $e = m -1 > 0$;
\item $e = 0, n = 1$, and $p_{1} > \min \{q_{1},\ldots, q_{m}\}$; or 
\item $e = 0, m = 1$, and $q_{1} > \min \{p_{1},\ldots, p_{n}\}$.
\end{enumerate}
We will prove the theorem by checking the first two case and the other cases will follow from Proposition \ref{new}.
\begin{enumerate}
\item If $ e > m - 1$, then $L$ has an equivalent description $M(e-m;(p_{1},1), \ldots, (p_{n},1), (q_{1}, q_{1} - 1) , \ldots, (q_{m}, q_{m}-1))$ with an alternating associated diagram. Therefore, the result in this case follows by Proposition \ref{basic}.
\item We show the result by induction on $m$.  
Now we assume that the result holds for $m = k$. We want to show that the result holds for $m = k + 1$. Take $q_{k + 1} = 3$, so $L  = P(k;p_{1}, \ldots, p_{n}, -q_{1}, \ldots, -q_{k}, -3)$ with $c(D) = k + p_{1} + \ldots + p_{n} + q_{1} + \ldots + q_{k} + 3$, where $D$ is the pretzel diagram of $L$. If we show that $c(D) \leq \det(L)$, then the result will follow by Theorem \ref{twist}. As a result \cite[Proof of Theorem.\,1.4]{G}, $L$ is quasi-alternating at the any crossing in the last tassel.  
We smooth one of the crossings in the last tassel 
\begin{align*}
\det(L) = \det(L_{0}) + \det(L_{1}) & \geq \sum_{i = 1}^{n} p_{i} + \sum_{j = 1}^{k} q_{j} + k + 2 + 1\\ & =  \sum_{i = 1}^{n} p_{i} + \sum_{j = 1}^{k} q_{j} + k + 3 = c(D),
\end{align*}
where the only inequality follows by Lemma \ref{just} for 
$L_{0} = P(k;p_{1}, \ldots, p_{n}, -q_{1}, \ldots, -q_{k}) = M(0;(p_{1},1),\ldots, (p_{n},1), (q_{1}, q_{1}-1) , \ldots, (q_{k}, q_{k}-1))$ 
and $L_{1} = P(k;p_{1}, \ldots, p_{n}, -q_{1}, \ldots, \\ -q_{k}, -2) = P(k - 1;2,p_{1}, \ldots, p_{n}, -q_{1}, \ldots, -q_{k})$ with $\det(L_{1}) \geq 1$ by induction if $k - 1 > 0$ and by Proposition \ref{new} if $ k - 1 = 0$.

\end{enumerate}
\end{proof}

\begin{thm}
The crossing number of any quasi-alternating Montesinos link mentioned in \cite[Theorem.\,1.2]{W} is less than or equal to its determinant.
\end{thm}
\begin{proof}
We prove this theorem by checking each case of the cases mentioned in \cite[Theorem.\,1.2]{W}. 
We let $L^{1} = L(a_{1}a_{2}, R , -n)$ with $1 + a_{1}(a_{2} - n) < 0, L^{2} = 	L(a_{1}a_{2}, R , (-c_{1})(-c_{2}))$ with $a_{2} < c_{2}$ or $a_{2} = c_{2}$ and $ a_{1} > c_{1}$, and $L^{3} = 	L(a_{1}a_{2}a_{3}, R , -n)$ with $a_{3} <  n$. We also denote $\widehat{L^{i}}$ to be the link $L^{i}$ in which $R$ is replaced by one single positive crossing where $ i = 1,2,$ and 3 for now and the rest of the proof. The author of \cite{W} shows that the link $\widehat{L^{i}}$ is quasi-alternating at the only single crossing in the middle tangle and applies \cite[Theorem.\,2.1]{CK} to show that $L^{i}$ is quasi-alternating at any crossing of the tangle $R$ that replaces the only single crossing in the middle tangle in $\widehat{L^{i}}$. We show that the result holds for $\widehat{L^{i}}$ and then apply Theorem \ref{twist} at the only single crossing in the middle tangle to obtain the result for $L^{i}$. Now we smooth the links $\widehat{L^{i}}$ at the only single crossing in the middle tangle to obtain

\begin{align*}
\det(\widehat{L^{1}}) = \det(\widehat{L^{1}_{0}}) + \det(\widehat{L^{1}_{1}}) & = \det(T(2,-n)\# C(a_{1}, a_{2})) + \det(\widehat{L^{1}_{1}}) \\
& \geq n(1 + a_{1}a_{2}) + 1 = n + na_{1}a_{2} + 1 \geq a_{1} + a_{2} + 1 + n  = c(\widehat{D^{1}}).
\end{align*}

\begin{align*}
\det(\widehat{L^{2}}) = \det(\widehat{L^{2}_{0}}) + \det(\widehat{L^{2}_{1}}) & = \det(C(a_{1},a_{2})\# C(-c_{1}, -c_{2})) + \det(\widehat{L^{2}_{1}}) \\
& \geq (1 + a_{1}a_{2})(1 + c_{1}c_{2}) + 1 = 1  + a_{1}a_{2} + c_{1}c_{2} + a_{1}a_{2}c_{1}c_{2}  + 1 \\ & \geq a_{1} + a_{2} + 1 + c_{1} c_{2} = c(\widehat{D^{2}}).
\end{align*}

\begin{align*}
\det(\widehat{L^{3}}) = \det(\widehat{L^{3}_{0}}) + \det(\widehat{L^{3}_{1}}) & = \det(T(2,-n)\# C(a_{1}, a_{2}, a_{3})) + \det(\widehat{L^{3}_{1}}) \\
& \geq n(a_{1} + a _{3} + a_{1}a_{2}a_{3}) + 1 = na_{1} + na_{3} + na_{1}a_{2}a_{3} + 1 \\ & \geq a_{1} + a_{2} + a_{3} + 1 + n  = c(\widehat{D^{3}}).
\end{align*}
The above inequalities hold by Lemma \ref{simple} and since $\widehat{L^{i}_{1}}$ is quasi-alternating with $ \det(\widehat{L^{i}_{1}}) \geq 1$.
\end{proof}

\begin{thm}
The crossing number of the quasi-alternating Montesinos link $ L = L(a_{1}a_{2}a_{3}, R, \\ (-c_{1})(-c_{2})(-c_{3}))$ with $\frac{1 + a_{1}a_{2}}{a + c_{1}c_{2}} >\frac{a_{1} + a_{3} + a_{1}a_{2}a_{3}}{c_{1} + c_{3} + c_{1}c_{2}c_{3}} $ is less than or equal to its determinant.
\end{thm}
\begin{proof}
We denote $\widehat{L}$ to be the link $L$ in which $R$ is replaced by one single positive crossing. The author of \cite{W} states that the link $\widehat{L}$ is quasi-alternating at the only single crossing in the middle tangle in \cite[Remark in page.\,8]{W} and applies \cite[Theorem.\,2.1]{CK} to show that $L$ is quasi-alternating at any crossing of the tangle $R$ that replaces the only single crossing in the middle tangle in $\widehat{L}$. We show that the result holds for $\widehat{L}$ and then apply Theorem \ref{twist} at the only single crossing in the middle tangle to obtain the result for $L$. Now we smooth the links $\widehat{L}$ at the only single crossing in the middle tangle to obtain

\begin{align*}
\det(\widehat{L}) = \det(\widehat{L_{0}}) + \det(\widehat{L_{1}}) & = \det( C(a_{1}, a_{2}, a_{3}) \# C(-c_{1},-c_{2},-c_{3})) + \det(\widehat{L_{1}}) \\
& \geq (a_{1} + a_{3} + a_{1}a_{2}a_{3})(c_{1} + c_{3} + c_{1}c_{2}c_{3}) + 1 \\ & \geq  (a_{1} + a_{3} + a_{2})(c_{1} + c_{3} + c_{2}) + 1 \\& \geq a_{1} + a_{2} + a_{3} + c_{1} + c_{2} + c_{3} + 1 = c(\widehat{D}).
\end{align*}
The above inequalities hold by Lemma \ref{simple} since $\widehat{L_{1}}$ is quasi-alternating with $ \det(\widehat{L_{1}}) \geq 1$.
\end{proof}
\section{Conjecture and Closing Argument and remarks} We close this with the following conjecture.
\begin{conj}\label{main}
For any quasi-alternating link $L$, we have \begin{equation}\label{one}
c(L) \leq \det(L).
\end{equation}
\end{conj}
In support to this conjecture, it holds for alternating links as a special class of quasi-alternating links by Proposition \ref{basic}. Also, the author of  \cite{G} in Proposition 3.2 proves the above conjecture in very special cases where the determinant is less than or equal 3.

The above argument in the previous section shows that inequality \ref{one} holds for infinitely many quasi-alternating links that are not alternating links.

This conjecture is not a characterization of quasi-alternating links since the converse of this conjecture is not true. In particular, the knots $9_{46}$ and $11n 139$ satisfy inequality \ref{one} and they are not quasi-alternating. The above knots are not quasi-alternating because Shumakovitch shows 
in \cite{S} that the knot $9_{46}$ has torsion in its odd Khovanov homology groups and Green in \cite[Theorem.\,3.1]{G} proves that the knot $11n 50$ does not bound a negative definite 4-manifold with torsion-free.

Also in support of the above conjecture, we provide a table of all knots up to 11 crossings using \cite{CL} that don't satisfy inequality \ref{one}. Also, we checked by hand that every knot appears in that table is not quasi-alternating since each knot is homologically thick in rational Khovanov homology using \cite{BM} except for the knots $10_{140}$ and $11n 139$ which are pretzel knots of the form $P(p,3,-3)$ with $p = 4$ or 5 (see table \ref{1}).
Moreover, we use \cite{CL} to provide two tables the first one of knots with 12 crossings and the second one of links up to 9 or less crossings that we conjecture to be not quasi-alternating based on inequality \ref{one} (see table \ref{2} and table \ref{3} respectively).

\begin{table}[h]
\centering
\begin{tabular}{|c|c||c|c||c|c|}
\hline
Knot & Determinant & Knot & Determinant & Knot & Determinant  \\
\hline
$8_{19}$ & 3 & $9_{42}$ & 7 & $10_{124}$ & 1 \\
\hline
$10_{132}$ & 5 & $10_{139}$ & 9 & $10_{140}$  & 9 \\
\hline
$10_{145}$ & 3 & $10_{153}$ & 1 & $10_{161}$ & 5  \\
\hline
$11n 9$ & 5 &  $11n 19$ & 5 & $11n 31$& 3 \\
\hline
$11n 34$ & 1 & $11n 38$ & 3 & $11n 42$& 1 \\
\hline
$11n 49$ & 1 & $11n 57$ & 7 & $11n 67$ & 9 \\
\hline
$11n 73$ & 9 & $11n 74$ & 9 & $11n 96$ & 7 \\
\hline
$11n 97$ & 9 & $11n 102$ & 3 & $11n 104$ & 3 \\
\hline
$11n 111$ & 7 & $11n 116$ & 1 & $11n 135$ & 5 \\
\hline
$11n 139$ & 9 & $11n 143$ & 9 & $11n 145$ & 9 \\
\hline
\end{tabular}

\vspace{0.2cm}
\caption{Knot table}
\label{1}
\end{table}

\begin{table}[h]
\centering
\begin{tabular}{|c|c||c|c||c|c|}
\hline
Knot & Determinant &  Knot & Determinant & Knot & Determinant \\
\hline
$12n 0019$ & 1 & $12n 0268$ & 9 & $12n 0473$& 1 \\
\hline
$12n 0023$ & 9 & $12n 0273$ & 5 & $12n 0475$& 7 \\
\hline
$12n 0025$ & 11 & $12n 0292$ & 1 & $12n 0487$ & 11 \\
\hline
$12n 0031$ & 9 & $12n 0293$ & 7 & $12n 0488$ & 5 \\
\hline
$12n 0051$ & 9 & $12n 0309$ & 1 & $12n 0502$ & 9 \\
\hline
$12n 0056$ & 9 & $12n 0313$ & 1 & $12n 0519$ & 7 \\
\hline
$12n 0057$ & 9 & $12n 0318$ & 1 & $12n 0552$ & 9 \\
\hline
$12n 0093$ & 11 & $12n 0321$ & 11 & $12n 0574$ & 9 \\
\hline
$12n 0096$ & 7 & $12n 0332$ & 9 & $12n 0575$ & 3 \\
\hline
$12n 0115$ & 11 & $12n 0336$ & 5 & $12n 0579$ & 9 \\
\hline
$12n 0118$ & 7 & $12n 0352$ & 7 & $12n 0582$ & 9 \\
\hline
$12n 0121$ & 1 & $12n 0355$ & 11 & $12n 0591$ & 7 \\
\hline
$12n 0124$ & 7 & $12n 0370$ & 5 & $12n 0605$ & 9 \\
\hline
$12n 0129$ & 7 & $12n 0371$ & 11 & $12n 0617$ & 5 \\
\hline
$12n 0138$ & 11 & $12n 0374$ & 11 & $12n 0644$ & 7 \\
\hline
$12n 0149$ & 5 & $12n 0386$ & 9 & $12n 0648$ & 11 \\
\hline
$12n 0175$ & 3 & $12n 0402$ & 9 & $12n 0655$ & 3 \\
\hline
$12n 0199$ & 11 & $12n 0403$ & 9 & $12n 0673$ & 5 \\
\hline
$12n 0200$ & 9 & $12n 0404$ & 3 & $12n 0676$ & 9 \\
\hline
$12n 0210$ & 1 & $12n 0419$ & 3 & $12n 0689$ & 7 \\
\hline
$12n 0214$ & 1 & $12n 0430$ & 1 & $12n 0725$ & 5 \\
\hline
$12n 0217$ & 5 & $12n 0433$ & 11 & $12n 0749$ & 7 \\
\hline
$12n 0221$ & 9 & $12n 0439$ & 3 & $12n 0812$ & 9 \\
\hline
$12n 0242$ & 1 & $12n 0446$ & 7 & $12n 0815$ & 7 \\
\hline
$12n 0243$ & 5 & $12n 0457$ & 11 & $12n 0851$ & 5 \\
\hline
\end{tabular}
\vspace{0.2cm}
\caption{Knot table}
\label{2}
\end{table}

\begin{table}[h]
\centering
\begin{tabular}{|c|c||c|c||c|c|}
\hline
Link & Determinant &  Link & Determinant & Link & Determinant \\
\hline
$L8n3$ & 4 &  $L9n4$ & 4  & $L9n18$ & 2   \\
\hline
$L8n6$ & 0 & $L9n9$ & 4   & $L9n19$ & 0  \\
\hline
$L8n8$ & 0 & $L9n12$ & 4  & $L9n21$ & 4 \\
\hline
$L9n3$ & 8 &  $L9n15$ & 2  &   $L9n27$ & 0\\
\hline
\end{tabular}
\vspace{0.2cm}
\caption{Link table}
\label{3}
\end{table}

The following remark shows how useful this conjecture if we can prove it for general quasi-alternating links.
\begin{rem}
If the above conjecture is true.
\begin{enumerate}
    \item This is an easy obstruction for quasi-alternateness compared to the other obstructions mentioned in the introduction.
    \item  It would solve the conjecture mentioned in \cite{G} that states that there are only finitely many quasi-alternating links with a given determinant.
    \item It will imply \cite[Theorem.\,2]{GW} and \cite[Proposition.\,3]{GW}. In particular, it will show that there are finitely many Kanenobu's knots that are quasi-alternating.
\end{enumerate}
\end{rem}

Finally, we hope to prove this conjecture in future work.
\section{Acknowledgment}
The first author thanks the Abdus Salam international centre for theoretical Physics and the Max Planck Institute for Mathematics for the kind hospitality during the course of this work. Also, the second author thanks Prof. Thomas Zasalvelsky for helpful communications.

\end{document}